\documentclass[reqno]{amsart}
\usepackage{graphicx, enumerate,amsmath,amsthm, amssymb,hyperref,amscd}
\newtheorem{theorem}{Theorem}[section]

\newtheorem{lemma}[theorem]{Lemma}
\newtheorem{proposition}[theorem]{Proposition}

\theoremstyle{definition}

\theoremstyle{remark}

\numberwithin{equation}{section}
\newcommand{\norm}[1]{\left\Vert#1\right\Vert}

\newcommand{\abs}[1]{\left\vert#1\right\vert}
\newcommand{\Norm}[1]{{\left\vert\kern-0.25ex\left\vert\kern-0.25ex\left\vert #1 
    \right\vert\kern-0.25ex\right\vert\kern-0.25ex\right\vert}}
\newcommand{\rl}{{\mathbb{R}}}
\newcommand{\cx}{{\mathbb{C}}}

\newcommand{\tensor}{\otimes}

\newcommand{\img}{\mathrm{img}}

\newcommand{\dbar}{\overline{\partial}}
\newcommand{\dom}{\mathrm{Dom}}

\newcommand{\tmop}[1]{\ensuremath{\operatorname{#1}}}
\renewcommand{\Re}{\tmop{Re}}

\newcommand{\oh}{\mathbb{O}}
\newcommand{\p}{\mathbb{P}}
\newcommand{\e}{\mathbb{E}}
\title{The $L^2$-cohomology of a bounded smooth Stein Domain is not necessarily Hausdorff}
\thanks{Debraj Chakrabarti  was partially supported by a grant from the Simons Foundation (\#316632), the Indo-US Virtual Institute for Mathematical and Statistical Sciences (VI-MSS),  and  an Early Career internal grant from Central Michigan University. Mei-Chi Shaw was partially supported by  National Science Foundation grants  DMS-1101415
and DMS-1362175}
\subjclass[2010]{32W05, 32V25, 32C35}
\author{Debraj Chakrabarti}
\address{Department of Mathematics, Central Michigan University, Mt. Pleasant,  MI 48859,  USA}
\email{chakr2d@cmich.edu}
\author{Mei-Chi Shaw}
\address{Department of Mathematics,  University of Notre Dame, Notre Dame, IN 46656, USA. }
\email{mei-chi.shaw.1@nd.edu }
\begin{document}
\begin{abstract} We give an example of a pseudoconvex  domain in a complex manifold whose $L^2$-Dolbeault cohomology  is non-Hausdorff, yet the domain is Stein.   The domain   is a smoothly bounded Levi-flat domain in a two complex-dimensional compact complex manifold.
The domain is  biholomorphic to a product domain in $\cx^2$, hence Stein. This implies that for $q>0$,  the usual Dolbeault cohomology with respect to smooth forms vanishes in degree $(p,q)$.
But the $L^2$-Cauchy-Riemann operator  on the domain does not have closed range on   $(2,1)$-forms and consequently its  $L^2$-Dolbeault cohomology is not Hausdorff. 
\end{abstract}

\maketitle

\section{Introduction}
For each bidegree $(p,q)$, with  $p\geq 0, q>0$, the Dolbeault Cohomology group $H^{p,q}(\Omega)$  of a Stein manifold $\Omega$ in degree $(p,q)$ vanishes,
and indeed this property characterizes Stein manifolds among complex manifolds 
(see e.g. \cite{GuRo, Ho1}).  In particular, with respect to the Fr\'{e}chet topology,
the operator $\dbar$ from the space $\mathcal{A}^{p,q-1}(\Omega)$ of smooth $(p,q-1)$-forms on $\Omega$ to the space $\mathcal{A}^{p,q}(\Omega)$
has closed range, since this range coincides with the null-space of the operator $\dbar:\mathcal{A}^{p,q}(\Omega)\to \mathcal{A}^{p,q+1}(\Omega)$. The aim 
of this paper is to show that things are much more interesting for the $L^2$-cohomology of a bounded Stein domain in a Hermitian manifold.

Recall that a {\em Hermitian manifold} is a complex manifold whose tangent bundle has been endowed with a Hermitian metric.
On a Hermitian manifold, one can define the $L^2$-Dolbeault Cohomology groups, which 
capture the $L^2$-function theory on the manifold.  If the manifold $\Omega$ is realized as a relatively compact (i.e. bounded) domain inside a larger Hermitian manifold
$X$ (with the restricted metric),   the $L^2$-spaces of forms and functions and the $L^2$-cohomology groups
do not depend on   the particular choice of the metric on $X$. According to a famous theorem of H\"{o}rmander, the $L^2$-Dolbeault groups 
$H^{p,q}_{L^2}(\Omega)$
vanish for $q>0$, provided $\Omega$ is a bounded Stein  domain in a Stein manifold.
 In particular, the range of the $\dbar$ operator in the $L^2$-sense
is closed in each degree.  One can  also show that if $\Omega$ is a smoothly bounded domain  with strongly pseudoconvex boundary   in  a complex Hermitian manifold,
then again $\dbar$ has closed range, and the $L^2$-Dolbeault groups are finite dimensional. 

The question arises whether    on a bounded Stein domain in a general complex manifold, the $\dbar$-operator still has closed range,
or equivalently whether the $L^2$-cohomology groups  of such a domain are Hausdorff in the natural quotient topology. This is not the case:
\begin{theorem} \label{thm-main} There is a compact complex surface $X$ and a relatively compact, smoothly bounded, Stein domain $\mathbb{O}$ in $ X$,
 such that the range of the $L^2$ $\dbar$-operator from the space $L^2_{2,0}(\oh)$ of square integrable $(2,0)$-forms to the space 
$L^2_{2,1}(\oh)$ of the square integrable $(2,1)$-forms on $\oh$ is not closed. Consequently, the $L^2$-cohomology space $H^{2,1}_{L^2}(\oh)$ is not Hausdorff in its natural quotient topology. 
\end{theorem}

The complex surface $X$ in the above theorem can be taken to be the product $\p^1\times \e$, where  $\p^1$ is the  projective line, and $\e$ is an elliptic curve defined by a  rectangular 
lattice in $\cx$. The domain $\oh$    defined below in Section 4  is  a domain introduced by Ohsawa 
in \cite{ohsawa}, which has the  
property that in spite of the fact that its boundary is smooth (even real analytic), the domain $\oh$ is biholomorphic to a product of two planar domains, one of which 
is the punctured plane $\cx^*=\cx\setminus\{0\}$, and the other is a bounded annulus.
Though $\oh$ is a Stein domain, its properties are quite unlike what one would expect from a bounded pseudoconvex domain in a Stein manifold.
The domain  $\oh$ used here as a counterexample, as well as closely related domains, have been studied before in several contexts (see \cite{Ba0, do, Ba1}).

The main tool in the proof of Theorem~\ref{thm-main} is the study of holomorphic extension of CR functions defined on the boundary of 
$\oh$. On a general complex manifold $M$, there is no Hartogs-Bochner phenomenon (see \cite{KR} or \cite{Lu}). 
 In other words,  it is not true that a CR function defined on the connected boundary $b\Omega$ of a smoothly 
bounded domain $\Omega\Subset M$ can be extended holomorphically into $\Omega$. Such Hartogs-Bochner extension does take place when the ambient $M$ is $\cx^n, n\geq 2$, or $M$ is Stein with 
dimension at least 2. In Section~\ref{sec-cr} below, we study some obstructions 
to holomorphic extension of CR functions, and relate them to the non-closed range property for the $\dbar$-operator. 
Suppose that the range of $\dbar$ is closed in $L^2_{2,1}(\oh)$. Using the  $L^2$ version of   Serre duality    (see  \cite{l2serre}),  the $L^2$ cohomology for $(2,1)$-forms  $H^{2,1}_{L^2} (\oh)$ is isomorphic to   
$H^{0,1}_{c,L^2}(\oh)$, the $L^2$ cohomology  for $(0,1)$-forms with minimal realization of the $\dbar$-operator (see Section~\ref{sec-defnot} below for definitions).    
One of the main observations in this paper is  that  if  the range of $\dbar$ is closed in $L^2_{2,1}(\oh)$, then 
$H^{0,1}_{c,L^2}(\oh)$ vanishes.  Then the $\dbar$-Neumann operator exists on $L^2_{2,1}(\oh)$ and   $H^{2,1}_{L^2} (\oh)$ vanishes. In this case, one can show that 
  the $\dbar$-Neumann operator is regular from $W^\epsilon$ to itself for some small $\epsilon>0$. Then from the $L^2$  $\dbar$-Cauchy problem, every CR  function on the  boundary $b\oh$  in $W^{\frac{1}{2} -\epsilon} (b\oh)$  can be  extended  holomorphically to $\oh$.  On the other hand,  the boundary of $\oh$ is Levi-flat and contains two complex tori, which divide the boundary into two disjoint parts.  If we take two distinct constants on the disjoint boundary pieces, it defines a    
CR function on the boundary of $\oh$ which is in $W^{\frac{1}{2} -\epsilon} (b\oh)$, but they do not extend (see Proposition 3.1).

There are well-known elementary examples of non-Stein domains in $\cx^2$ where closed range of the $\dbar$-operator  does not hold either in the  Fr\'{e}chet topology (cf. \cite[Section~14]{serre}) or in the $L^2$-topology (cf. \cite[Page 75--76]{fk}).
There is also  an example by  Malgrange of a bounded
domain with pseudoconvex boundary  (in a two complex dimensional torus)  for which the $\dbar$-operator does not have closed range in the Fr\'{e}chet sense (cf. \cite{malgrange}).
Similar phenomena happen in some noncompact complex Lie groups (cf. \cite{kazama}). 
As already noted, in a Stein manifold, in the Fr\'{e}chet topology on the space of smooth forms  the $\dbar$-operator has closed range. On the other hand, in the $L^2$-sense, many unbounded domains in $\cx^n$ (for example
the whole of $\cx^n$ itself) are easily seen not to have the closed range property for the $\dbar$-operator. 
The example of $\oh$ considered here shows that such non-closed range phenomena are also possible on bounded Stein domains with
smooth boundary. This is in sharp contrast with the case for bounded domains  in $\Bbb C^n$.   In fact, it has been shown recently that for a bounded Lipschitz domain  in $\cx^2$ such that the complement of the domain is connected in $\cx^2$, the $\dbar$ equation has closed range in $L^2_{0,1}(\Omega)$ if and only if the domain is pseudoconvex  (cf. \cite[Theorem~3.4]{LS}).

The range of the map   $\dbar:L^2_{2,0}(\oh)\dashrightarrow L^2_{2,1}(\oh)$ is not closed, and  it is  desirable to find an explicit  $\dbar$-closed  $L^2$  $(2,1)$-form  which is not in the $L^2$  range.
In   Section~\ref{sec-zonly},   we show that there is a  harmonic $(2,1)$-form  $g$ on $X$ such that $g|_{\oh}$  is not  in the range of $\dbar$ in the $L^2$ sense, but we do not know whether $g|_{\oh}$ is in the closure of the range.
On the other hand, there exists a solution $v$ of $\dbar  v=g$  which is in $W^{-\epsilon}(\oh)$ for each $\epsilon>0$.
Recalling that  as a complex manifold, $\oh$ is the product of the punctured plane $\cx^*$ with an annulus, one may suspect that the lack of closed range of $\dbar$ is somehow connected to the unboundedness of the 
punctured plane. In Section~\ref{sec-zonly}, we also show that this is not the case by looking at the $L^2$-function theory on $\oh$ in more detail, and in particular show that for $(2,0)$-forms which 
depend only on the $\cx^*$ factor, the $\dbar$ operator actually has closed range. This shows that the non-closed range property is more subtle. 

\bigskip
\noindent
{\em Acknowledgements:} Debraj Chakrabarti would like to  thank  K. Sandeep for helpful discussions on fractional Sobolev spaces.   Mei-Chi Shaw would like to thank
David   Barrett for pointing out his paper \cite {Ba0} which inspires the present work.  Both authors  would like to thank Sophia Vassiliadou for her comments on  a previous version of the manuscript, and the referee for numerous helpfull suggestions.

 \section{Definitions and notation}\label{sec-defnot}
We briefly recall the definitions related to the $L^2$-Dolbeault cohomology. Details on the $L^2$-theory of the $\dbar$-equation, and the associated $\dbar$-Neumann problem may be found
in standard texts on the subject, e.g. \cite{Ho1, fk,chen-shaw,straube}, and details on the $L^2$-version of Serre duality theorem used in this paper may be found  in \cite{l2serre}.

Let $\Omega$ be a Hermitian manifold, i.e., a complex manifold with a Riemannian metric which is Hermitian with respect to the complex structure on the tangent spaces. 
We can define in a natural way the $L^2$-spaces of $(p,q)$-forms $L^2_{p,q}(\Omega)$. In the special case  when $\Omega$ is realized as a relatively compact domain in a larger complex manifold $X$, 
and is given a Hermitian metric by restricting from $X$, the spaces $L^2_{p,q}(\Omega)$ are defined independently of the particular choice of the metric in $X$. In our application, the domain 
$\Omega$ will be of this latter type.

The differential operator $\dbar$ is defined classically on the space $\mathcal{A}^{*,*}(\Omega)$ of smooth forms on $\Omega$, and maps $(p,q)$-forms to $(p,q+1)$-forms.
Standard techniques of functional analysis allow us to construct extensions of the operator $\dbar$ acting as closed unbounded operators on the spaces $L^2_{p,q}(\Omega)$.
Two such realizations of $\dbar$ as an unbounded operator on $L^2_{*,*}(\Omega)$ are of fundamental importance. First, the (weak) {\em maximal realization} of $\dbar$ is defined 
to have domain $\mathcal{D}_{\rm max}^{p,q}=\{f\in L^2_{p,q}(\Omega)\mid \dbar f \in L^2_{p,q+1}\}$, where $\dbar$ acts on an $L^2$-form in the sense of distributions. 
By standard abuse of notation we will denote by $\dbar$ the unbounded, closed, densely Hilbert space operator from $L^2_{p,q}(\Omega)$ to $L^2_{p,q}(\Omega)$ with domain
$\mathcal{D}_{\rm max}^{p,q}$, and such that for $f\in \mathcal{D}_{\rm max}^{p,q}$, we take $\dbar f$ in the sense of distributions. We also define the (strong) {\em  minimal realization}
$\dbar_c$ as the closed operator with smallest domain $\mathcal{D}_{\rm min}^{p,q}$ which extends the restriction of the $\dbar$-operator to the smooth compactly supported forms. This may be interpreted 
by saying that forms in $\mathcal{D}^{p,q}_{\rm min}$ satisfy a boundary condition.
The existence of such a ``closure''{} $\dbar_c$ is proved in functional analysis. Note that $\mathcal{D}_{\rm max}^{p,q}\supsetneq \mathcal{D}_{\rm min}^{p,q}$ on a bounded domain $\Omega$, 
and the maximal realization $\dbar$ is an extension of the operator $\dbar_c$. Note that $\dbar\circ\dbar =\dbar_c\circ\dbar_c =0$, so we have two different examples of 
{\em Hilbert Complexes }(cf. \cite{hilcom}), i.e., a chain complex (in the sense of homological algebra) in which the differential operators are unbounded operators on Hilbert spaces, and the chain groups
are the domains of these operators. One can associate to such a complex its cohomology, and this way we obtain, for any Hermitian manifold $\Omega$, its  {\em $L^2$-Dolbeault cohomology} groups:
\[ H^{p,q}_{L^2}(\Omega) = \frac{\ker(\dbar :L^2_{p,q}(\Omega)\dashrightarrow L^2_{p,q+1}(\Omega))}{{ \img}(\dbar :L^2_{p,q-1}(\Omega)\dashrightarrow L^2_{p,q}(\Omega))},\]
where the dashed arrows signify that the maximal realization is defined only on a dense subspace of the Hilbert space $L^2_{*,*}(\Omega)$ (in this case $\mathcal{D}^{*,*}_{\rm max}$). Similarly, we can define the 
{\em $L^2$-Dolbeault cohomology with minimal realization} by setting
\[ H^{p,q}_{c, L^2}(\Omega) = \frac{\ker(\dbar_c :L^2_{p,q}(\Omega)\dashrightarrow L^2_{p,q+1}(\Omega))}{{ \img}(\dbar_c :L^2_{p,q-1}(\Omega)\dashrightarrow L^2_{p,q}(\Omega))},\]
where now the minimally realized Hilbert space operators $\dbar_c$ are used. In general, the $L^2$-Dolbeault groups, and the $L^2$-Dolbeault groups with minimal realization are very different, but under appropriate
hypotheses, there is a relation of duality between the two collections of groups, which is an analog of the classical Serre duality in the $L^2$-setting (see \cite{l2serre} for details).

Note that the cohomology spaces $H^{p,q}_{L^2}(\Omega) $ and $H^{p,q}_{c, L^2}(\Omega)$ are complex vector spaces, and have a natural linear topology as quotients of linear topological spaces. In particular,
for $H^{p,q}_{L^2}(\Omega) $, if the range ${\img}(\dbar :L^2_{p,q-1}(\Omega)\dashrightarrow L^2_{p,q}(\Omega))$ is a  closed subspace of $L^2_{p,q}(\Omega)$, then $H^{p,q}_{L^2}(\Omega) $ has the natural structure of a Hilbert space itself.
However, if ${ \img}(\dbar :L^2_{p,q-1}(\Omega)\dashrightarrow L^2_{p,q}(\Omega))$ is not closed, the quotient topology is not even Hausdorff. The closed range property for the $\dbar$-operator has other important consequences,
and most importantly, it is equivalent to the possibility of solving the $\dbar$-problem with $L^2$-estimates.

The Hilbert space adjoints $\dbar^*$ and $\dbar_c^*$ are again closed, densely defined, unbounded operators on $L^2_{*,*}(\Omega)$. The {\em Complex Laplacian } is the operator $\Box=\dbar\dbar^*+\dbar^*\dbar$,
and its inverse (modulo kernel) is the $\dbar$-Neumann operator ${\mathsf N}$. Both map the space $L^2_{p,q}(\Omega)$ to itself for each degree $(p,q)$.  The kernel of $\Box$ in degree $(p,q)$ consists of the  {\em Harmonic
forms.} More details on these constructions may be found in the texts mentioned above.

\section{Extension of CR functions from boundaries of manifolds}
\label{sec-cr}

\begin{proposition}\label{prop-tool}
Let $X$ be a  Hermitian manifold of complex dimension $n$, and let $\Omega$ be a smoothly bounded relatively compact domain 
in $X$. Suppose that
\begin{enumerate}[(a)]
\item $b\Omega$ is connected, and there is a smooth complex hypersurface $H$ (not necessarily connected)  of $X$ such that $H\subset b\Omega$ and $b\Omega\setminus H$ is not connected.
\item The $L^2$-cohomology in degree (0,1) with minimal realization vanishes, i.e.,
\[ H^{0,1}_{c,L^2}(\Omega)=0.\]
\end{enumerate}
Then the $L^2$ $\dbar$-operator from $L^2_{n,n-2}(\Omega)$ to $L^2_{n,n-1}(\Omega)$ does not have closed range. Consequently,
the  $L^2$-cohomology group $ H^{n,n-1}_{L^2}(\Omega)$
is not Hausdorff in its natural topology.
\end{proposition}

\begin{proof} For a contradiction, assume that the $\dbar$-operator from $L^2_{n,n-2}(\Omega)$ to $L^2_{n,n-1}(\Omega)$ has closed range.
We claim that then there is an $\epsilon$ with $0<\epsilon<\dfrac{1}{2}$ such that  {\em each CR function $f$ on $b\Omega$  which belongs to  the $L^2$-Sobolev space $W^\epsilon(b\Omega)$ of order $\epsilon$ extends to
 holomorphic function $F$ on $\Omega$.} (In the Sobolev context, ``extension''{} means that $F\in W^{\frac{1}{2}+\epsilon}(\Omega)$, 
and the Sobolev trace of $F$ on $b\Omega$ is $f$). Postponing the proof of the claim for now, to produce a contradiction it suffices to produce a CR function on $b\Omega$ which is of class $W^s$ on $b\Omega$, 
for  each $0<s<\dfrac{1}{2}$, and which does not admit a holomorphic extension to $\Omega$. But such a function is easy to construct as follows.

Let $b\Omega^+$ be a connected component of $b\Omega\setminus H$, and let  $b\Omega^- $ be the union of the remaining connected 
components.
By hypothesis, neither of $b\Omega^\pm$
is empty, their union is $b\Omega\setminus H$, and the closures of  $b\Omega^+$ and $b\Omega^-$ meet along the complex 
hypersurface $H$. 
Define a locally integrable function $f$ on $b\Omega$ by setting $f\equiv1 $ on $b\Omega^+$ and $f\equiv 0$ on $b\Omega^-$. 

We claim that  $f$ is CR on $b\Omega $ in the sense of distributions. This is clear except along the complex hypersurface $H$, where $f$ has
a jump discontinuity. To verify the fact that $f$ is CR on $H$, let $p\in H$, and choose a local $\mathcal{C}^\infty$ real  coordinate system $(x_1,y_1,\dots, x_{n-1},y_{n-1},t)$ on 
a neighborhood of $p$ in $b\Omega$ such that the hypersurface $H$ is given by $\{t=0\}$, and along $H$, $z_j=x_j+iy_j$ is a complex coordinate on $H$ for $1\leq j \leq n-1$.
 We can further assume that $f=1$ on $\{t>0\}$ and $f=0$ on $\{t<0\}$. At the point $p$, a basis of $(0,1)$-vector fields is given by the $(n-1)$ vectors $\dfrac{\partial}{\partial \overline z_j}$, where $1\leq j \leq n-1$. 
It follows that in a neighborhood of $p$ in $b\Omega$  there are vector fields 
 \[Z_j= \frac{\partial}{\partial \overline{z}_j}+ \sum_{\ell=1}^{n-1}\left(a^\ell_j\frac{\partial}{\partial x_\ell}+b^\ell_j\frac{\partial}{\partial y_\ell}\right)+c \frac{\partial}{\partial t},\]
which pointwise span the CR vector fields, where $a^\ell_j, b^\ell_j$ and $c$ (with $1\leq j,\ell \leq n-1$) are smooth functions near $p$ which vanish identically on the hypersurface $H$.
To show that $f$ is CR, we need to check that  for each $j$, we have  $Z_jf=0$ in the sense of distributions (in order to induce a distribution corresponding to the function $f$, we use the standard 
volume form of $\rl^{2n-1}$).  For each $\phi$ which is smooth and compactly supported near $p$, we need to show that $(Z_jf)(\phi)=0$, 
 i.e., $\int_{t>0}{Z_j^t\phi}=0$, 
where $Z_j^t= -\dfrac{\partial}{\partial \overline{z}_j}+ Y_j+ \lambda_j $, where $Y_j$ is a smooth vector field and $\lambda_j$ a smooth function given by
\begin{align}\label{eq-y1} Y_j&= - \left(\sum_{\ell=1}^{n-1}\left(a^\ell_j\frac{\partial}{\partial x_\ell}+b^\ell_j\frac{\partial}{\partial y_\ell}\right)+c \frac{\partial}{\partial t}\right)\\
\lambda_j&=-
\left(\sum_{\ell=1}^{n-1}\left(\frac{\partial{a_j^\ell}}{\partial x_\ell}+\frac{\partial{b_j^\ell}}{\partial y_\ell}\right)+\frac{\partial{c}}{\partial t}\right).\label{eq-lambda}\end{align}


Now,  assuming that the support of $\phi$ is contained in the cube $\{ \abs{x_j}<\delta, \abs{y_j}<\delta, \abs{t}<\delta\}$, we have using Fubini\rq{}s theorem
(the hat  $\widehat{\cdot}$  in the formulas below indicates that this factor is absent from the product):
\begin{align*}
\int_{t>0}{ \frac{\partial \phi}{\partial\overline{ z}_j}}
& = \int\left( \int_{\abs{x_j}<\delta, \abs{y_j}<\delta} \frac{\partial{ \phi}}{\partial\overline{ z}_j}dx_jdy_j\right) dx_1dy_1\ldots\widehat{dx_j}\widehat{dy_j}\ldots dx_{n-1}dy_{n-1}dt\\
&= \int_{t=0}^\delta 0 \cdot dx_1dy_1\ldots\widehat{dx_j}\widehat{dy_j}\ldots dx_{n-1}dy_{n-1}dt\\
&=0,
\end{align*}
where the inner integral in the repeated integral vanishes  by an application of the divergence formula.
To complete the proof we need to show that $\int_{t>0} (Y_j +\lambda_j)\phi=0$.
Now let $\psi$ be a compactly supported smooth function  on the real line which is identically 1 in a neighborhood of 0, and  for small $\epsilon>0$, let $\phi_\epsilon(x_1,y_1,\ldots,x_{n-1},y_{n-1},t)= \phi(x_1,y_1,\ldots,x_{n-1},y_{n-1},t)\psi(\frac{t}{\epsilon})$. Writing
$\phi = \phi_\epsilon + (\phi-\phi_\epsilon)$, we see that $\int_{t>0} (Y_j+\lambda_j)  (\phi-\phi_\epsilon)=0$, since $f$ is a smooth CR function outside $H$. 
The integral $\int_{t>0}\lambda_j \phi_\epsilon$ is clearly $O(\epsilon)$ as $\epsilon\to 0$, since the function $\lambda_j\phi_\epsilon$   is bounded uniformly in $\epsilon$, and the integral ranges over a
subset of the support of $\phi_\epsilon$, which has volume $O(\epsilon)$.  Noting that the coefficients $a_j^\ell, b_j^\ell,c$ vanish along $H$, we see that they are $O(\epsilon)$ in the support of $\phi_\epsilon$. Further, $\abs{\dfrac{\partial}{\partial t}(\phi_\epsilon)}=O\displaystyle{\left(\frac{1}{\epsilon}\right)}$.
Therefore, 
\begin{align*}
\abs{\int_{t>0}{Y_j\phi_\epsilon}}&\leq \sup\abs{\sum_{\ell=1}^{n-1}\left(a^\ell_j\frac{\partial\phi_\epsilon}{\partial x_\ell}+{b^\ell_j}\frac{\partial\phi_\epsilon}{\partial y_\ell} \right)+{c}\frac{\partial\phi_\epsilon}{\partial t}}
\cdot{\rm vol}({\rm support}\phi_\epsilon)
\\
&\leq C(\epsilon +\epsilon + \epsilon\cdot \frac{1}{\epsilon})\cdot \epsilon\\
&= O(\epsilon).
\end{align*}

Combining the estimates, it follows that for each $\epsilon>0$, we have  that 
\begin{align*}
\abs{(Z_jf)(\phi)} &= \abs{\int_{t>0}Z_j^t \phi}\\
&= \abs{ \int_{t>0} -\dfrac{\partial\phi}{\partial \overline{z}_j}+ Y_j\phi+ \lambda_j\phi}\\
& \leq C\epsilon,
\end{align*}
so that $Z_jf=0$ for $1\leq j \leq n-1$, i.e., $f$ is CR in the sense of distributions.

We now claim that  for $0<s<\dfrac{1}{2}$, 
the CR function $f$ belongs to the fractional-order Sobolev space $W^{s}(b\Omega)$.
Let $B$ denote the unit ball of $\rl^{2n-1}$ and let $B^\pm$ denote the upper and the lower half balls. To see that $f\in W^s(b\Omega)$, 
it clearly suffices to show $u\in W^s(B)$, where $u$ is the function on $B$ which is identically 1 on $B^+$ and which vanishes on $B^-$. The fact that 
$u\in W^s(B)$  for $0<s<\dfrac{1}{2}$ can be verified by a direct computation, cf.  \cite[Proposition~5.3]{tay}, or it follows from 
 \cite[Theorem~11.4]{lima1}, where it is shown that the extension by 0 of a function in $W^s(B^+)$  is a function in $W^s(B)$ if $0<s<\frac{1}{2}$.

Note that the function $f$ does not admit a holomorphic extension to either $\Omega$ or $X\setminus\Omega$. If such an extension $F$
did exist, by standard estimates on  Bochner-Martinelli type singular integrals,  $F$ would be $\mathcal{C}^\infty$-smooth 
up to the boundary on $b\Omega\setminus H$. Then
 by a classical boundary uniqueness
result, the holomorphic function $F$ has to be  simultaneously both identically 1 and identically 0 on $\Omega$.  This produces the required contradiction. 
To complete the proof we only need to establish the claim made in the first paragraph regarding the existence of an $\epsilon$ with $0<\epsilon<\dfrac{1}{2}$ for which each CR function of class $W^\epsilon(b\Omega)$ extends holomorphically 
into $\Omega$.

By assumption, the $\dbar$-operator has closed range in $L^2_{n,n-1}(\Omega)$. Then the  $\dbar$-Neumann operator
$\mathsf{N}=\mathsf{N}_{n,n-1}$ exists on $\Omega$ as a bounded operator on $L^2_{n,n-1}(\Omega)$, since in the top degree,
$\dbar$ automatically has closed range from $L^2_{n,n-1}(\Omega)$ to $L^2_{n,n}(\Omega)$.  Further, by $L^2$-Serre duality (cf. \cite[Theorem~2]{l2serre}) it follows that $H^{n,n-1}_{L^2}(\Omega)=0$, and further that the harmonic space $\mathcal{H}^{n,n-1}(\Omega)=0$ (cf. \cite{l2serre}).
Using an observation of Kohn (see Proposition~\ref{prop-kohn} below),  there 
is an $0<s<\dfrac{1}{2}$ such that 
the operator $\mathsf{N}$ extends as a bounded operator on $W^s_{n,n-1}(\Omega)$, i.e., the space of  $(n,n-1)$-forms on $\Omega$ with 
coefficients in $W^s(\Omega)$. When $\mathsf{N}$ maps $W^s$ to itself, it follows that the ``canonical solution operator''{} $\dbar^*\mathsf{N}$
and all the related operators ($\dbar \mathsf{N}$ and the Bergman projection)
also map $W^s$ to itself. (For integral $s$, a proof based on a standard commutator estimate may be found in  texts on the $\dbar$-Neumann problem, see
\cite[Theorem~6.1.4 and Theorem~6.2.2]{chen-shaw} or \cite[Lemma~3.2 and Corollary~3.3]{straube}.  Essentially the same arguments generalize to
 fractional  $s$, even with $0<s<1$, if we replace the tangential differential operators in the commutator estimates by tangential pseudo-differential operators
of fractional order and use the commutator estimates for such operators, see \cite[Appendix, Prop.~A.5.1 to A.5.3]{fk}). Therefore $\dbar \mathsf{N}$  
and  $\dbar^* \mathsf{N}$  map  $W^s(\Omega)$ to itself. If $0<s<\frac{1}{2}$, the space  $\mathcal{C}^\infty_0(\Omega)$ is dense in $W^s(\Omega)$, so that $W^s(\Omega)$ 
coincides with $W^s_0(\Omega)$ (cf. \cite[Theorem~11.1]{lima1}), and therefore $W^{-s}(\Omega)$, defined to be the dual of $W^s_0(\Omega)$, is 
actually the dual of $W^s(\Omega)$. Since $\mathsf{N}$ is self-adjoint, it maps $W^{-s}(\Omega)$ to itself.   It  follows that the  associated operator   $\dbar\mathsf{N}$ (or  $\ \dbar^*\mathsf{N} $)  maps $W^{-s}(\Omega)$ to itself.

  Now let $f$ be a CR function on $b\Omega$ belonging to $W^\epsilon(b\Omega)$
where $\epsilon= \frac{1}{2}-s$. We can extend $f$ to a function $\tilde{f}\in W^{\frac{1}{2}+\epsilon}(\Omega)$. Then $\dbar \tilde{f}\in W^{\epsilon-\frac{1}{2}}(\Omega)= W^{-s}(\Omega)$. We consider the function
\[ u_c = - \ast \dbar \mathsf{N} \left(\ast \dbar \tilde{f}\right),\]
where $\ast$ denotes the Hodge Star operator on the space of differential forms on the  Hermitian manifold $\Omega$. Since $\ast$ induces an isometry of each Sobolev space, 
  it follows that $u_c \in W^{-s}(\Omega)$. It follows from \cite[Theorem~3]{l2serre}, that $u_c$ satisfies $\dbar_c u_c=\dbar \tilde{f}$, where $\dbar_c$ is the minimal realization of the $\dbar$ operator (see \cite{l2serre} or     the proof of Theorem 3.2 in  \cite{Sh}).
   Furthermore, using the regularity of $\dbar \mathsf{N}$ on $W^{-s}(\Omega)$ established above, we have
\[ \norm{u_c}_{W^{-s}(\Omega)}\leq C \norm{\dbar \tilde{f}}_{W^{-s}(\Omega)} \leq C \norm{\tilde{f}}_{W^{1-s}(\Omega)}\leq C\norm{f}_{W^{\frac{1}{2}-s}(b\Omega)}.\]
If we define
\[ F=\tilde{f}-u_c,\]
then $F\in W^{-s}(\Omega)$ and $F$ is holomorphic. Using the weak and strong extension of $\dbar_c$   (cf. \cite[Lemma ~2.4]{LS}), 
  we see that 
$u_c$ satisfies the equation $\dbar u_c = \dbar \tilde f$ in an open  neighbourhood of $\overline \Omega$ if we set $u_c$ and $\dbar \tilde f$ equal to zero outside $\Omega$.  From the interior regularity for $\dbar$, we have that $u_c$ is in $W^{1-s}(\overline \Omega)$.  This shows that $u_c$ actually has trace in $W^{\frac12-s}(b\Omega)$. The trace of $u_c$ on $b\Omega$    is equal to zero since $u_c$ satisfies the $\dbar$-Cauchy problem.    Noting that the distributional boundary value of $F$ on $b\Omega$ is of class $W^{\frac 12-s}(b\Omega)$,
it follows that $F\in W^{1-s}(\Omega)$, so that $F$   is indeed a holomorphic extension of $f$.   
\end{proof}
To complete the argument in the previous result, we will use the following observation, which is  due to Kohn (see \cite{BS2}). 
\begin{proposition}\label{prop-kohn}Let $\Omega$  be a relatively compact domain with smooth boundary  in a Hermitian manifold $X$. Suppose that for some degree $(p,q)$,  the $\dbar$-Neumann
operator $\mathsf{N}_{p,q}$ exists on $\Omega$ as a bounded linear operator on the space $L^2_{p,q}(\Omega)$ of square integrable $(p,q)$-forms, and suppose that the harmonic space of degree $(p,q)$ is trivial. Then 
there is an $\epsilon_0$, with $0<\epsilon_0\leq \frac{1}{2}$, such that for $0<\epsilon<\epsilon_0$, the operator $\mathsf{N}_{p,q}$ extends as a bounded operator from the Sobolev space
 $W^{\epsilon}_{(p,q)}(\Omega)$ to itself.
\end{proposition}
If $\Omega$   admits a bounded H\"{o}lder continuous plurisubharmonic exhaustion 
(which happens when $X=\cx^n$ \cite{difo} or $X$ is a complex projective space \cite{OS,CSW}) then the conclusion follows directly. However, 
the existence of such a bounded plurisubharmonic exhaustion is not necessary for the conclusion to hold. In our application,
we will use it on  a domain which does not admit a bounded plurisubharmonic exhaustion (see Lemma~\ref{lem-3}).  The proof of the proposition is the same as the arguments used in Kohn-Nirenberg \cite{koni2} using pseudo-differential operators (see \cite{koni65}) and we omit the details. 


\section{The   proof of Theorem~\ref{thm-main}}
\label{sec-ohsawa}
\subsection{The domain $\oh$}As was already mentioned in the introduction, the domain $\oh$ will be realized as a domain with real-analytic boundary in the  compact complex 
surface $X=\p^1\times \e$, where $\p^1$ is the complex projective line (the Riemann Sphere) and $\e$ is an elliptic curve defined in the following way.
Let $\beta>0$, and let $\Gamma_\beta=\{e^{2k\pi\beta} | k\in \mathbb{Z}\}$ be the cyclic subgroup of $\cx^*=\cx\setminus\{0\}$ generated by the number
$e^{2\pi\beta}$. Then $\e$ is the quotient
$\cx^*/\Gamma_\beta$ which  is an elliptic curve.

Let the natural projection $\cx^*\to \e$ be denoted by $w\mapsto [w]$ (so that $[e^{2\pi\beta}w]=[w]$) and let  $z:\p\to\cx\cup\{\infty\}$ denote the inhomogeneous coordinate on $\p$.
Then $\oh$ is the domain in  $\p\times \e$ given  by
\begin{equation}\label{eq-ohdefn}
\oh = \left\{(z,[w])\in \p\times\e \colon \Re(zw)>0\right\},
\end{equation}
where it is easily seen that the condition $\Re(zw)>0$ is well defined independently of the choice of the lift $w$ of the point $[w]\in \e$.
This is a smoothy bounded Levi-flat domain in $X=\p\times\e$, and is in fact biholomorphic to a product domain in $\cx^2$,
where one factor is $\cx^*$ and the other is an annulus. Indeed, let 
\begin{equation}\label{eq-annulus}
\mathbb{A} =\left\{W\in \cx \colon e^{-\frac{\pi}{2\beta}}<\abs{W}<e^{\frac{\pi}{2\beta}}\right \},
\end{equation}
which is an annulus in the plane. Let $\Phi$ be the map from $\cx^*\times \mathbb{A}$  to $X$ given by
\begin{equation}\label{eq-Phi} \Phi(Z,W)= \left(Z, \left[Z^{-1}\exp\left(i\beta\cdot \log W\right)\right]\right)= \left(Z, \left[Z^{-1}\cdot W^{i\beta}\right]\right).\end{equation}
$\Phi$ is well-defined in spite of the multivaluedness of the logarithm.  It is not difficult to verify that 
 $\Phi$ is a biholomorphism from $\cx^*\times \mathbb{A}$ onto $\oh$. We
will refer to $\Omega=\cx^*\times \mathbb{A}$ as the {\em product model} of $\oh$. 

Note that the domain $\oh$ depends on the choice of the parameter $\beta$, and the domains obtained for distinct $\beta$ are easily seen to be non-biholomorphic. Therefore, in fact we have a one-parameter family of counterexamples
to prove Theorem~\ref{thm-main}. In the sequel, we consider the Ohsawa domain corresponding to one fixed $\beta$.

 The Levi structure of $b\oh$ can be summarized as follows:
\begin{proposition} \label{prop-levistructure}
$b\oh$ is a smooth, real-analytic, connected, Levi-flat hypersurface. The complex tori $\{0\}\times \e$ and $\{\infty\}\times \e$ are contained in $b\oh$ and the complement of these two tori
is a disjoint union of two open subsets $\Sigma^\pm$ of $b\oh$. Each of $\Sigma^\pm$ is CR equivalent to  the product $\cx^*\times S^1$  (with the natural CR structure).
\end{proposition}
\begin{proof} The assertions in the first sentence can be verified by direct computation, starting from the representation \eqref{eq-ohdefn}.

The map $\Phi$  extends biholomorphically to a neighborhood of $\cx^*\times \overline{\mathbb{A}}$, and it is easy to see that the  image of $\cx^*\times b{\mathbb{A}}$ is all of 
$b\oh$ except the two tori $\{0\}\times \e$ and $\{\infty\}\times \e$. Note   that the boundary of the annulus $\mathbb{A}$ consists
of two circles $b\mathbb{A}^+$ and $b\mathbb{A}^-$. Define $\Sigma^+= \Phi(\cx^*\times b\mathbb{A}^+ )$ and $\Sigma^-= \Phi(\cx^*\times b\mathbb{A}^- )$. 
Then $b\oh$ is the disjoint union of the four pieces  $\Sigma^+$, $\Sigma^-$,  $\{0\}\times \e$ and $\{\infty\}\times \e$. By construction 
 each of $\Sigma^\pm$ is a Levi-flat hypersurface 
biholomorphically equivalent to $\cx^*\times S^1$ (where $S^1$ is the circle).
\end{proof}

If we use $(Z,W)$ as coordinates on $\oh$,
each function  holomorphic on $\oh$ has a Laurent expansion
\begin{equation}\label{eq-laurent} \sum_{(j,k)\in \mathbb{Z}^2} a_{jk}Z^j W^k. \end{equation}

Viewed  as a complex manifold, being a product of planar domains,  $\oh$ is Stein, and
the Dolbeault Cohomology  $H^{p,q}(\oh)$ vanishes for each positive $q$.  Note further there are nonconstant bounded holomorphic
functions on $\oh$, namely, the ones represented in the product model as bounded holomorphic functions of $W$ alone.  

In order to study $L^2$ theory on $\oh$, we need to impose an arbitrary Hermitian metric on $X$ and restrict it to $\oh$.
The actual $L^2$-spaces of forms and functions, the realizations of the $\dbar$-operator are independent of the choice of the metric.
For simplicity therefore we give the most symmetric metric to $X$, which arises as the product metric of constant curvature metrics on the factors.
 We endow $\p$ (which is diffeomorphic
to the round 2-sphere)  with a round metric (the Fubini-Study metric)  and the elliptic curve $\e$ with a flat metric.  The metric on $\p$  is normalized such that 
\begin{equation}\label{eq-varphi} \phi= \frac{dz}{1+\abs{z}^2}, \end{equation}
is a $(1,0)$-form of unit length at each point. Denote by $\psi$ the unique $(1,0)$-form on $\e$ whose pullback to $\cx^*$ by the map $\pi:w\mapsto [w]$
satisfies
\begin{equation}\label{eq-psi} \pi^*\psi= \frac{dw}{w}.\end{equation}
Such a $\psi$ exists since the form $\dfrac{dw}{w}$ is periodic with respect to the action of the group $\Gamma_\beta$ by multiplication on $\cx^*$.
Then $\psi$ is in fact a holomorphic 1-form
on $\e$, and by declaring it to be of unit length we get a flat Hermitian metric on $\e$. Therefore, the  metric on $\oh$  induced from $\p\times\e$
is represented as
\begin{equation}\label{eq-metric} ds^2=\phi\tensor \overline{\phi}+\psi\tensor\overline\psi,\end{equation}
where, by standard abuse of notation, we denote also by $\phi$ and $\psi$ the pullbacks 
of these forms from $\p$ and $\e$ respectively to the product $\p\times\e$ via  projections on the factors.

We can pull back the metric \eqref{eq-metric} on $\oh$ via the map $\Phi$ and obtain a metric $\Phi^*(ds^2)$ on the product model $\Omega=\cx^*\times \mathbb{A}$,
so that with these metrics, $\Phi$ becomes an isometry.
By a direct computation we can verify that the Riemannian volume form on $\Omega$ associated to the pullback metric  $\Phi^*(ds^2)$ is represented as:
\begin{equation}\label{eq-volform} \omega = \frac{\beta^2}{\left(1+\abs{Z}^2\right)^2\abs{W}^2}dV\end{equation}
where $dV=(-2i)^{-2}dZ\wedge d\overline{Z}\wedge dW\wedge d \overline{W}$ represents the Euclidean volume form on $\cx^2$. 
 We characterize some Sobolev spaces of holomorphic functions: 
\begin{lemma}\label{lem-oo}
(1) The Bergman space $\mathcal{O}(\oh)\cap L^2(\oh)$ consists of functions which are represented in the product model $\cx^*\times \mathbb{A}$ as functions on $\mathbb{A}$ alone, with no dependence on 
$\cx^*$.\\
(2) The Holomorphic Sobolev space $\mathcal{O}(\oh)\cap W^1(\oh)$ consists of only the  constant functions. \\
(3) The above statements remain true if $\oh$ is replaced by $X\setminus \overline \oh$.
\end{lemma}
\begin{proof}The proofs are by direct  computations. For (1), it suffices to note that in the product representation,  thanks to the structure  of the volume form \eqref{eq-volform}, the terms of the Laurent series
\eqref{eq-laurent} are orthogonal. Consequently, it suffices to prove that terms of the form $Z^jW^k$ are not in $L^2$ if $j\not=0$. But this follows from integration on the product representation $\cx^*\times \mathbb{A}$ using the
volume form \eqref{eq-volform}.

Thanks to part (1) it suffices to show that the functions $W^k$, where $k\not=0$ are not in the Sobolev space $W^1(\oh)$. To show this, it suffices 
to show that the differential $d(W^k)=kW^{k-1}dW$ is not in the space $L^2_{1,0}(\oh)$. 	Since $W^{i\beta}=zw$, differentiating logarithmically, we see that
\begin{align*}
i\beta \frac{dW}{W}&= \frac{dw}{w}+ \frac{dz}{z}\\
&=\psi + \frac{1+\abs{z}^2}{z}\phi,
\end{align*}
where $\psi$ and $\phi$ are as in \eqref{eq-psi} and \eqref{eq-varphi} respectively.
Using $z=Z$ and the pointwise orthogonality of $\psi$ and $\phi$, we have for the pointwise norm:
\[ \beta^2 \frac{\abs{dW}^2}{\abs{W}^2}=1+ \frac{\left(1+\abs{Z}^2\right)^2}{\abs{Z}^2}.\]
Using the volume form \eqref{eq-volform}, we now see that $W^{k-1}dW$ is not square integrable for any $k$, so that it follows that the function $W^k$ is not in the Sobolev space $W^1(\oh)$. Therefore, the only functions 
in $\mathcal{O}(\oh)\cap W^1(\oh)$ are the constants.

The last statement follows since $\oh$ and $X\setminus \overline\oh$ are isometrically biholomorphic by the map $(z,w)\to (-z,w)$.

\end{proof}

The following lemma shows that in spite of the boundedness of the domain $\oh$ in $\p\times \e$, in some respects the function theory on $\oh$ is analogous to that on an unbounded domain:
\begin{lemma}\label{lem-3}   (a)  Bounded plurisubharmonic functions on $ {\oh}$ do not separate points. In particular, $\oh$ does not admit a bounded plurisubharmonic exhaustion function.\\
(b) A continuous plurisubharmonic function on $\overline{\oh}$ is a constant.
\end{lemma}
\begin{proof} For (a), representing $\oh$ as a product, we see that any bounded plurisubharmonic function on $\oh$ must be constant on each slice $\cx^* \times \{W\}$ for each $W\in \mathbb{A}$. 

For (b), note that the restriction of such a continuous plurisubharmonic function to the torus $T_0=\{0\}\times \e$ is a constant.  But it is easy to see that $T_0$ is contained in the closure of each slice $\cx^* \times \{w\}$  of the 
product representation. Consequently, the constant value assumed by the plurisubharmonic function on each slice is the same.
\end{proof}

\noindent{\bf Remark:} The   function $\Re(zw)$  is pluriharmonic  (hence plurisubharmonic), and serves as a defining function for the domain $\oh$ except near the torus $T_\infty=\{\infty\}\times \e$.  
But it is not a global defining function  of $\oh$ since it is not defined near $T_\infty$. 
A related  result on the non-existence of a plurisubharmonic exhaustion function  on a certain Stein domain  with Levi-flat boundary  may be found in  in \cite[Theorem~1.2]{OS}.


\subsection{Proof of Theorem~\ref{thm-main}}
We can now complete the proof of Theorem~\ref{thm-main}.  Let 
\[ H= \left( \{0\}\times \e\right) \cup \left(\{\infty\}\times \e\right).\]
It is easy to see that  $b\oh$ is actually connected, and  
 $b\oh\setminus H$ is not connected, since it is the disjoint union of $\Sigma^+$ and $\Sigma^-$.

To complete the proof of Theorem~\ref{thm-main} we will study the cohomology  group   $H^{0,1}_{c,L^2}(\oh)$. This $L^2$-cohomology group is 
independent of the choice of the metric adopted, so we can simplify our work by choosing the metric \eqref{eq-metric}. It suffices to endow the ambient
manifold $X=\p\times \e$ with a Hermitian metric, which then can be restricted to the domain $\oh$.
We now prove the following:
\begin{lemma} \label{lem-4.4}With the  metric \eqref{eq-metric} (and therefore with any other comparable metric),  suppose that  the range of $\dbar:L^2_{2,0}(\oh)
\to L^2_{2,1}(\oh)$ is closed in $L^2_{2,1}(\oh)$. Then we have
\[ H^{0,1}_{c,L^2}(\oh)= 0\]
\end{lemma}
\begin{proof} Let $f$ be a $\dbar_c$-closed $(0,1)$-form on $\oh$. We need to show that the equation $\dbar_c u=f$ has a solution. 
We instead first consider the equation on $X=\p\times\e$ given as
$\dbar \tilde{u}=\widetilde{f}$,
where $\widetilde{f}$ is the  $\dbar$-closed  $(0,1)$-form  obtained by extending the form $f$ as 0 outside $\oh$. By the Hodge decomposition on 
the compact K\"ahler manifold $X$,
this equation has a solution
provided $\widetilde{f}$ is orthogonal to the Harmonic space $\mathcal{H}^{0,1}(X)$. But by the K\"unneth formula, $\mathcal{H}^{0,1}(X)$ is one dimensional and generated by 
the form $\overline{\psi}$, where $\psi$ is as  in \eqref{eq-psi}, and so there is a $\widetilde{u}$ satisfying the equation provided  $(\widetilde{f},\overline{\psi})_X=0$.

Define a smooth  (0,2)-form on $\oh$ by setting $v={z}{\overline\phi}\wedge{\overline\psi}$. One easily sees that $v$ is not 
in $L^2_{0,2}(\oh)$.  Recall that, if $\vartheta$ denotes the formal adjoint of the differential operator $\dbar$, we can write $\vartheta=- \ast \dbar\ast$, where $\ast$ denotes the Hodge star operator
on the de Rham complex of $X$, which is a $\cx$-antilinear map  (see \cite{l2serre} for details), and we have
\begin{align*}\vartheta v &= -\ast \dbar \ast (z\overline{\phi}\wedge\overline{\psi})\\
&= \ast \dbar(\overline{z}\phi\wedge\psi)\\
&=\ast \dbar\left(\frac{\overline{z}}{1+\abs{z}^2}d{z}\wedge\psi\right)\\
&=\ast \left(\frac{{1}}{(1+\abs{z}^2)^2}d\overline{z}\wedge dz\wedge\psi\right)\\
&=\ast( \overline\phi\wedge {\phi}\wedge\psi)\\
&=\overline{\psi.}
\end{align*}
Notice that if it were true that $v\in L^2_{0,2}(\oh)$ then we would have
$ (\widetilde{f},\overline{\psi})_X = (f, \vartheta  v)_{\oh}= (\dbar f, v)_{\oh}=0,$
where the boundary term in the integration by parts vanishes since $f$ is in the domain of $\dbar_c$. Since $v$ is not in $L^2$, we need a more delicate argument to prove the claim.

  Since $f$ satisfies $\dbar_c f=0$ in $\oh$,  using Friedrichs' lemma   we can approximate  $f$ by a sequence $\{f_\nu\}$ of smooth forms  with compact support in $\oh$ such that $f_\nu\to f$ in $L^2_{0,1}(\oh)$  and $\dbar f_\nu\to \dbar f=0$ in $L^2_{0,2}(\oh)$ as $\nu\to\infty$.  
    Thus we have 
\begin{align*} (\widetilde{f},\overline{\psi})_X&= (f,\vartheta  v)_{\oh}\\&=\lim_{\nu\to \infty} (f_\nu, \vartheta v)_{\oh}
\\& =\lim_{\nu\to \infty} (\dbar f_\nu, v)_{\oh},
\end{align*}
where the boundary term vanishes since $f_\nu$ has compact support in $\oh$. 

Assuming that $\dbar :L^2_{2,0}(\oh) \to L^2_{2,1}(\oh)$ has closed range, this implies by duality that 
 $\dbar_c :L^2_{0,1}(\oh)\to L^2_{0,2}(\oh)$ has closed range (see Theorem 3 in  \cite{l2serre}).  Since  $\dbar_cf_\nu=\dbar f_\nu$ is in the range of $\dbar_c$, there exists $u_c^\nu$ with $u_c^\nu\in \text{Dom}(\dbar_c)$ such that $$\dbar_c u_c^\nu=\dbar f_\nu $$
 and 
 \begin{equation}\label{4.9} \|u_c^\nu\|\le C\|\dbar_c f_\nu\|\to 0.\end{equation}
If $\dbar$ has closed range in $L^2_{2,1}(\oh)$, then  
  the $\dbar$-Neumann operator   $N_{2,0} :L^2_{2,0}(\oh) \to L^2_{2,0}(\oh)  $ exists.  We can take 
$u_c^\nu= -\ast\dbar  N^\nu_{2,0}\ast \dbar f_\nu $  
and $u_c^\nu$ satisfies (4.9).  One  can approximate $u_c^\nu$ by a sequence of smooth forms   $u_{c}^{\nu,\epsilon}$  compactly supported in $\oh$   such that
$u_{c}^{\nu,\epsilon}\to u_c^\nu$ and  $\dbar u_{c}^{\nu,\epsilon}\to \dbar u_{c}^{\nu}$ in $L^2$ (see e.g. Lemma 2.4 in \cite{LS}). 
 
 Therefore
\begin{align*}
(\dbar f_\nu, v)_{\oh}&=  ( \dbar_c u_c^\nu, v)_{\oh}= \lim_{\epsilon\to 0} ( \dbar_c u_{c}^{\nu,\epsilon}, v)_{\oh}
\\&= \lim_{\epsilon\to 0} (   u_{c}^{\nu,\epsilon}, \vartheta v)_{\oh}
\\&=  (  u_c^\nu,  \overline \psi )_{\oh} 
\\& \to 0 
\end{align*}
where we have used (4.9). 
 Therefore $(\widetilde{f},\overline{\psi})_X=0$  and
this  shows that there is a $\widetilde{u}$ on $X$ satisfying $\dbar\tilde{u}=\tilde{f}$. Since $\tilde{f}\in L^2(X)$, by interior elliptic 
regularity, we have $\tilde{u}\in W^1(X)$. Since $\tilde{u}|_{X\setminus\overline\oh}$ is holomorphic, it follows from Lemma~\ref{lem-oo}, parts 2 and 3
that $\tilde{u}$ reduces to a constant $c$ on $X\setminus \oh$. Therefore $u=\tilde{u}-c$ is a   function in $X$
satisfying $\dbar u=f$ in $X$. The  support of $u$ is  in 
$\overline{\oh}$.  The boundary $b\oh$ is smooth.  Thus we have, using the weak and strong extension for $\dbar_c$   (cf. \cite[Lemma~2.4]{LS})  
  that  $u\in \text{Dom}(\dbar_c)$ and    $\dbar_c u =f$. 
\end{proof}

 \begin{proof}[End of Proof of Theorem~\ref{thm-main}] Suppose that 
  the range of $\dbar:L^2_{2,0}(\oh)
\to L^2_{2,1}(\oh)$ is closed in $L^2_{2,1}(\oh)$. Then by Lemma~\ref{lem-4.4} we have
$H^{0,1}_{c,L^2}(\oh)= 0$. But this is in  contradiction with the conclusion of Proposition~\ref{prop-tool}.    

\end{proof}

\section{Forms on $\oh$ depending on the $\cx^*$ factor only }
\label{sec-zonly}
Since $\oh$ is biholomorphic to $\cx^*\times \mathbb{A}$, it may seem that the lack of closed range proved above might somehow be related to the factor $\cx^*$. 
To investigate this, we consider forms on $\oh$, whose coefficients depend only on the variable $Z$ ranging over $\cx^*$ and not on the variable $W$ ranging over $\mathbb{A}$.
In the natural coordinates $(z,w)$ of $\p\times \e$, this corresponds to considering forms on $\oh$ whose coefficients are functions of the inhomogeneous coordinate $z$ on $\p\setminus\{\infty\}$,
since $z=Z$ by \eqref{eq-Phi}.

Let us say that a $(2,0)$-form $u\cdot dz\wedge \psi$ on $\oh$ {\em depends on $z$ only} if the coefficient function $u$ is a function of 
the coordinate $z$ alone. Denote by $L^2_{2,0}(\oh, {\text {$z$-only}})$ the space of square-integrable $(2,0)$-forms which depend on $z$ only. Similarly, let  $L^2_{2,1}(\oh, \text{$z$-{\rm only}})$ be the subspace 
of $L^2_{2,1}(\oh)$ consisting of forms of the type $v\cdot d\overline{z}\wedge dz\wedge\psi$, with $v$ a function of $z$. Note that $L^2_{2,0}(\oh, {\text {$z$-only}})$ and $L^2_{2,1}(\oh, \text{$z$-{\rm only}})$
are closed subspaces of $L^2_{2,0}(\oh)$ and $L^2_{2,1}(\oh)$ respectively.
Note further that each form in $L^2_{2,1}(\oh, \text{$z$-{\rm only}})$ is automatically closed, and 
the image under $\dbar$  of  $L^2_{2,0}(\oh, \text{$z$-{\rm only}})$  is contained in  $L^2_{2,1}(\oh, \text{$z$-{\rm only}})$.
Let   $T:L^2(\cx)\dashrightarrow L^2(\cx)$ be the  unbounded closed densely defined operator
given on smooth functions  by
\begin{equation}\label{eq-tdef} Tu(z)= (1+\abs{z}^2) \frac{\partial u}{\partial \overline{z}},
\end{equation}
and with  the weak maximal realization. Then we have the following:
\begin{lemma}\label{lem-commdiag} There are Hilbert-space isomorphisms $\eta: L^2_{2,0}(\oh, {\text {$z$-only}})\to L^2(\cx)$ and $\theta: L^2_{2,1}(\oh, \text{$z$-{\rm only}})\to L^2(\cx)$ such that the following diagram commutes:
\[\begin{CD}
L^2_{2,0}(\oh, \text{$z$-{\rm only}})\cap \dom(\dbar) @>\dbar>> L^2_{2,1}(\oh,\text{$z$-{\rm only}})\\
@V{\eta}VV  					@V{\theta}VV\\
\dom(T) @>T>> L^2(\cx)\\
\end{CD}\]
where $\dom(T)\subset L^2(\cx)$ is the domain of the maximally realized operator $T$, and the map $\eta: dom(\dbar)\cap L^2_{2,0}(\oh, \text{$z$-{\rm only}})\to \dom(T)$ is a bijection.
\end{lemma}
\begin{proof}
Let $\eta:L^2_{2,0}(\oh, \text{$z$-{\rm only}})\to L^2(\cx)$ be the map
\[ \eta : u\cdot dz\wedge \psi \mapsto u \]
and let $\theta: L^2_{2,1}(\oh, \text{$z$-{\rm only}})\to L^2(\cx)$ be the map 
\[ \theta: v \cdot\overline{\phi} \wedge \phi \wedge \psi \to \frac{v}{ (1+\abs{z}^2)}.\]
(Note that each element of $L^2_{2,1}(\oh, \text{$z$-{\rm only}})$ may be written in the form 
$ v \cdot\overline{\phi} \wedge \phi \wedge \psi $, for a function $v$ of $z\in \cx^*$). We first check that $\eta$ and $\theta$ 
are isomorphisms. The space $L^2_{2,0}(\oh)$ is defined independently of the metric, since its elements are sections of the canonical bundle, and we have
\begin{align*}
\norm{f}^2&= \int_\oh f\wedge \overline{f}\\
&= \int_{\oh}\abs{u}^2 dz\wedge \psi\wedge d\overline{z}\wedge \overline{\psi}\\
& =4\int_{\oh}\abs{u}^2 (1+\abs{z}^2)^2 d\oh,
\end{align*}
where $d\oh= (-2i)^{-2}\phi\wedge\overline{\phi}\wedge\psi\wedge\overline{\psi}$ is the Riemannian volume form of $\oh$. Pulling back the integral to $\cx^*\times\mathbb{A}$ via the map $\Phi$ of 
\eqref{eq-Phi}, and noting that under the pullback metric, the volume form is the form $\omega$ of \eqref{eq-volform}, we see that
\begin{align*}\norm{f}^2&=4 \int_{\cx^*\times\mathbb{A}}\abs{u}^2(1+\abs{Z}^2)^2 \omega\\
&=4\beta^2 \int_{\cx^*\times\mathbb{A}}\frac{\abs{u}^2}{\abs{W}^2}dV,
\end{align*}
where $dV$ denotes the standard Euclidean volume form of $\cx^2=\rl^4$. Therefore, if $f\in L^2_{2,0}(\oh,\text{$z$-only})$, it follows by Fubini'{}s theorem  that 
\[ \norm{f}^2_{L^2_{2,0}(\oh,\text{$z$-only})}= C\norm{u}^2_{L^2(\cx)},\]
where $C$ is a constant independent of $f=u\cdot dz\wedge\psi$. Since $\eta(f)= u$,  and $\eta$ is clearly surjective, it follows that $\eta$ is an isomorphism of Hilbert spaces.

We now consider the map $\theta: g\mapsto v\cdot(1+\abs{z}^2)^{-1}$, where $g=v\cdot\overline\phi \wedge \phi \wedge \psi $. We have
\begin{align*}
\norm{g}^2_{L^2_{2,1}(\oh)} &= \int_\oh \abs{v}^2 d\oh\\
&=\beta^2 \int_{\cx^*\times \mathbb{A}} \abs{v(Z)}^2 \frac{dV}{(1+\abs{Z}^2)^2\abs{W}^2}\\
&=C \int_{\cx^*}\frac{\abs{v(Z)}^2}{(1+\abs{Z}^2)^2}dV(Z)\\
&= C \norm{\theta(g)}^2_{L^2(\cx)},
\end{align*}
so that $\theta$ is also an isomorphism, since it is clear that $\theta$ is surjective.

Now let $f\in \dom(\dbar)\cap  L^2_{2,0}(\oh, \text{$z$-{\rm only}}) $, and write $f= u \cdot dz\wedge \psi$ where $u$ is a
function of the variable $z$ only. Then $T(\eta(f))= (1+\abs{z}^2)\dfrac{\partial u}{\partial \overline{z}}$. But
\begin{align*}\dbar f &=\dfrac{\partial u}{\partial \overline{z}} d\overline{z}\wedge  dz\wedge \psi\\
&=  (1+\abs{z}^2)^2\dfrac{\partial u}{\partial \overline{z}}\ \overline{\phi}\wedge{\phi}\wedge\psi,
\end{align*}
so that $\theta(\dbar f)=  (1+\abs{z}^2)\dfrac{\partial u}{\partial \overline{z}}= T(\eta f)$ and the diagram commutes. To see that $\eta$ is bijective, note first that $\eta$ is injective, since it is an isomorphism of Hilbert spaces. Further, it is clear that whenever $u\in \mathcal{C}^\infty(\cx)\cap \dom(T)$, then $u$ lies in the image of $\eta$. Now the surjectivity of 
$\eta$ follows from the density of smooth forms in the domain of the maximal  $L^2$-realization  of a differential operator.
\end{proof}
We now analyze the behavior of the operator $T$:

\begin{lemma}The range of the  operator $T:L^2(\cx)\dashrightarrow L^2(\cx)$ of \eqref{eq-tdef} is   the orthogonal complement of the function $\left(1+\abs{z}^2\right)^{-1}$ in $L^2(\cx)$. In particular
the range is closed.
\end{lemma}
\begin{proof}
We have the decomposition
$L^2(\cx)= \overline{{\rm img}(T)}\oplus \ker(T^*)$, where $T^*$ denotes the Hilbert space adjoint of $T$.   
Integration by parts shows that for $v\in {\rm Dom}(T^*)$, we have   $ T^*v = -\dfrac{\partial}{\partial z}\left\{(1+\abs{z}^2)v\right\}.$
Consequently, the kernel of $T^*$ consists of $v\in L^2(\cx)$ such that $(1+\abs{z}^2)\overline{v}$ is entire. Denoting this entire
function by $h$, we have  $\overline{v} =(1+\abs{z}^2)^{-1}h$ is in $L^2(\cx)$. Expanding $h$ in an entire Taylor series, we see that 
$\overline{v}\in L^2(\cx)$ if and only if $h$ is reduced to a constant.
Consequently, the space $\ker(T^*)$ is one-dimensional  and spanned by $(1+\abs{z}^2)^{-1}\in L^2(\cx)$.
Therefore
\[\overline{{\img}(T)}= \left\{f\in L^2(\cx)\colon \int_\cx {f}\cdot\left(1+\abs{z}^2\right)^{-1}dA=0\right\},\]
where $dA$ is the Lebesgue measure on $\cx$. Note also that $T$ is injective from $L^2(\cx)$ to $\img(T)$,
since its null-space consists of $L^2$ holomorphic functions  on $\cx$, the only instance of which is the zero function.

Let now $f\in \overline{{\img}(T)}$, i.e. the inner product $\left(f,(1+\abs{z}^2)^{-1}\right)$ vanishes. Noting that $f\in L^2(\cx)$, we easily conclude 
using the Cauchy-Schwarz inequality that $f \cdot\left(1+\abs{z}^2\right)^{-1}\in L^1(\cx)\cap L^2(\cx)$.
Define a function  $\gamma$ on $\cx$ by $\gamma= \displaystyle{\mathfrak{F}\left(f \cdot\left(1+\abs{z}^2\right)^{-1}\right)}$, where $\mathfrak{F}$ is the Fourier transform on $\mathbb R^2$. 
 By standard properties of 
the  Fourier transform, the function $\gamma$ is continuous, vanishes at infinity and is also in $L^2(\cx)$.  Furthermore, we have 
$(1+\Delta)\gamma=\hat{f}\in L^2(\cx)$, where $\Delta$ is the Laplacian, the operator with symbol $\abs{z}^2$.
Combining with the fact that $f\in L^2(\cx)$, we conclude that  $\Delta\gamma\in L^2(\cx)$, so that we have $\gamma\in W^2_{\rm loc}(\cx)$.
By Sobolev embedding then, $\gamma$ belongs locally  to the H\"{o}lder space $C^\alpha$ for any $0<\alpha<1$.  Also by definition of $\gamma$
\begin{align*} \gamma(0)&= \int_\cx (1+\abs{z}^2)^{-1} f(z) e^{-i\langle 0, z \rangle}dA(z)\\
&=\left(f, (1+\abs{z}^2)^{-1}\right)\\
&=0.
\end{align*}
It follows therefore that for each $\alpha$ with $0<\alpha<1$, there is a $C_\alpha$ such that 
\[ \abs{\gamma(\zeta)}\leq C_\alpha \abs{\zeta}^\alpha.\]

Let  $ u=\mathfrak{F}^{-1}\left(\dfrac{2\gamma}{\zeta}\right)$. We assume that $u$ is in $L^2(\cx)$ first. 
Then, noting that the Fourier multiplier corresponding to the operator $\dfrac{\partial}{\partial \overline{z}}$ is $\dfrac{\zeta}{2}$, we have \begin{align*} Tu&= (1+\abs{z}^2)\frac{\partial}{\partial{\overline{z}}}\mathfrak{F}^{-1}\left(\dfrac{2\gamma}{\zeta}\right)\\
&= (1+\abs{z}^2)\mathfrak{F}^{-1}(\gamma)\\
&= (1+\abs{z}^2)\cdot  (1+\abs{z}^2)^{-1} f\\
&=f,
\end{align*} so that $u$
satisfies  $Tu=f$. To complete the proof it suffices to show that
$u$  is in fact in $L^2(\cx)$, so that we will have $u\in \dom(T)$ and consequently $f\in {\img}(T)$.
We have for some constants $C>0$
 \begin{align*}\norm{\frac{\gamma}{\zeta}}_{L^2(\cx)}^2&\leq C\int_{\zeta\in \cx,|\zeta|\le 1} \frac{\abs{\zeta}^{2\alpha}}{\abs{\zeta}^2}dA(\zeta)+\int_{\zeta\in \cx,|\zeta|\ge 1} \frac{\abs{\gamma}^{2}}{\abs{\zeta}^2}dA(\zeta)\\
 &\leq C \left(\int_{r=0}^1{r^{2\alpha-1}}dr +1\right)\\
 &<\infty,
 \end{align*}
 where $A$ is a positive real number. This
  shows that $u= \mathfrak{F}^{-1}\left(\dfrac{2\gamma}{\zeta}\right)$ is in $L^2(\cx)$. This completes the proof.
\end{proof}
Then we have the following:
\begin{proposition}\label{prop-zonly} The operator $\dbar:L^2_{2,0}(\oh, {\text {$z$-{\rm only}}})\dashrightarrow L^2_{2,1}(\oh)$ has closed range. The range is the orthogonal complement of the form $\overline{\phi}\wedge\phi
\wedge \psi$ in the closed  subspace $ L^2_{2,1}(\oh,\text{$z$-{\rm only}})$.
\end{proposition}
\begin{proof}According to Lemma~\ref{lem-commdiag}, the range of the operator
$\dbar:L^2_{2,0}(\oh, \text{$z$-{\rm only}})\dashrightarrow L^2_{2,1}(\oh, \text{$z$-{\rm only}})$ is equal to
$\theta^{-1}\circ T \circ \eta$, and since $\eta$ is surjective onto $\dom(T)$, the range of $\dbar$ coincides with 
$\theta^{-1}({\rm range}(T))$. Since $\theta$ is an isomorphism and an isometry (up to a constant)  of Hilbert spaces,  and the range of $T$ is the orthogonal 
complement of $v=(1+\abs{z}^2)^{-1}$ in $L^2(\cx)$, it follows that the range of $\dbar$ is the orthogonal complement in $ L^2_{2,1}(\oh,\text{$z$-{\rm only}})$ of $\theta^{-1}(v)=\overline{\phi}\wedge\phi\wedge\psi$.
\end{proof}


The result of  Proposition~\ref{prop-zonly} can be strengthened to prove the following:

\begin{proposition}The form $g=\overline{\phi}\wedge\phi\wedge\psi$ is not in the range of $\dbar: L^2_{2,0}(\oh)\to L^2_{2,1}(\oh)$. 
\end{proposition}
\begin{proof} In view of the previous proposition, it is sufficient to prove that if there is a solution $  \mu$ in $ L^2_{2,0}(\oh)$ satisfying the equation
\[ \dbar  \mu=g = \overline{\phi}\wedge\phi\wedge\psi\]
then there is also a solution $\lambda\in L^2_{2,0}(\oh, \text{$z$-{\rm only}})$.

We work in the product model $\cx^*\times \mathbb{A}$ with natural coordinates $(Z,W)$. Then we have
\[ g = \frac{d\overline{Z}\wedge dZ \wedge dW}{(1+\abs{Z}^2)^2 W},\]
and suppose that in these coordinates the solution $\mu$ is given as
$\displaystyle{ \mu= u(Z,W) \frac{dZ}{1+\abs{Z}^2} \wedge \frac{dW}{W}},$
where from $\dbar\mu =g$ we conclude that
\[ \frac{\partial}{\partial\overline{Z}}\left( \frac{u(Z,W)}{1+\abs{Z}^2}\right) = \frac{1}{(1+\abs{Z}^2)^2},\]
from which we get the relation
\[ u_{\overline{Z}}\cdot(1+\abs{Z}^2) - Z\cdot u =1.\]
Since $\mu$ is in $L^2$, it follows that
\[ \int_{\cx^*\times\mathbb{A}} \frac{ \abs{u(Z,W)}^2}{(1+\abs{Z}^2)^2\abs{W}^2}dV <\infty,\]
where $dV$ is the volume form of $\cx^2$.
Let $v$ be defined as a function of $Z$ by
\[ v(Z)=\frac{1}{\abs{\mathbb{A}}} \int_{\mathbb{A}}u(Z,W) dA(W)\]
where $dA$ is Lebesgue measure on the plane.
Note that 
\begin{align*}
\frac{\partial}{\partial \overline{Z}} \left( \frac{v(Z)}{1+\abs{Z}^2}\right) &= \frac{v_{\overline{Z}}\cdot(1+\abs{Z}^2) - Z\cdot v }{(1+\abs{Z}^2)^2}\\
&= \frac{1}{(1+\abs{Z}^2)^2}\cdot \frac{1}{\abs{\mathbb{A}}}\cdot \int_{\mathbb{A}}\left(u_{\overline{Z}}\cdot(1+\abs{Z}^2) - Z\cdot u \right)dA\\
&=\frac{1}{(1+\abs{Z}^2)^2}.
\end{align*}
By the Cauchy-Schwarz inequality, $\abs{v(Z)}^2 \leq \abs{\mathbb{A}}^{-1} \int_{\mathbb{A}}\abs{u(Z,W)}^2 dA(W)$. Integrating this over
$\mathbb{\cx^*}$ it follows that $v\cdot (1+\abs{Z}^2)^{-1}\in L^2(\cx)$.

 We set
\[ \lambda = v(Z) \frac{dZ}{1+\abs{Z}^2} \wedge \frac{dW}{W},\]
which is a (2,0)-form. From the fact that  $v\cdot (1+\abs{Z}^2)^{-1}\in L^2(\cx)$, it follows that $\lambda \in L^2_{2,0}(\oh, \text{$z$-only})$.
Further,
\[ \dbar \lambda = \frac{\partial}{\partial \overline{Z}}\left( \frac{v(Z)}{1+\abs{Z}^2}\right) d\overline{Z}\wedge dZ \wedge \frac{dW}{W}= g.\]
The claim is proved.
\end{proof}

\bigskip
\noindent
{\bf Remarks:} 
\smallskip
\noindent
(1) 
In spite of Theorem~\ref{thm-main} and  Proposition~\ref{prop-zonly}, we do not have a full understanding of the $L^2$ theory of the
manifold $\oh$. In particular, it would be interesting to compute the  $L^2$-cohomology of $\oh$ in other degrees, and to know whether the form $g$ of the preceding 
lemma lies in the closure of the range of $\dbar$. 
In the earlier examples  by Serre \cite{serre} or Malgrange \cite{malgrange}, the groups $H^{2,1}$, $H^{1,1}$  and $H^{0,1}$ are 
isomorphic. 

\smallskip
\noindent
(2) We note that for a bounded pseudoconvex domain with smooth boundary  in   $\p^n$,  there does not  exist any hypersurface which 
satisfies the condition (a) in  Proposition~\ref{prop-tool} (see   \cite[Theorem~1.3]{do1}).


\end{document}